\documentclass[a4paper,12pt]{amsart}
\usepackage[letterpaper, margin=1.2in]{geometry}
\usepackage{latexsym}
\usepackage{amssymb}
\usepackage{amsmath}
\usepackage{amsfonts}
\usepackage[table]{xcolor}
\usepackage{pgfplots}
\usepackage{color}
\usepackage{txfonts,pxfonts} 
\usepackage{tikz}
\usetikzlibrary{calc,arrows}
\usetikzlibrary{calc}
\usepackage{verbatim}
\newtheorem{thm}{Theorem}[section]

\newtheorem{lem}[thm]{Lemma}

\theoremstyle{definition}

\newtheorem*{rem}{Remark}

\def\fph{\mathbb{F}_{\ph}}

\newcommand{\Z}{\mathbb Z}
\newcommand{\z}{\mathbb Z}
\newcommand{\Q}{\mathbb Q}

\newcommand{\F}{\mathbb F}
\newcommand{\fp}{\mathbb F_p}

\def\F{\mathbb{F}}

\newcommand{\p}{\mathfrak{p}}
\def\ol{\overline}
\def\al{\alpha}
\def\la{\lambda}

\def\th{\theta}
\def\md#1{\ \mbox{\rm(mod }{#1})}
\def\nph#1{N_{\ph}(#1)}
\def\npp#1{N_{\ph}^+(#1)}
\def\ph{\phi}

\newcounter{cs}
\stepcounter{cs}
\newcommand{\casos}{\begin{itemize}}
\newcommand{\fcasos}{\end{itemize}\setcounter{cs}{1}}

\newfont{\tit}{cmr12 scaled \magstep3}

\begin{document}

\title[]{On Power integral bases for certain pure number fields}
\author{Lhoussain El Fadil}
\address{Faculty of Sciences Dhar El Mahraz, P.O. Box  1874 Atlas-Fes , Sidi mohamed ben Abdellah University,  Morocco}\email{lhouelfadil2@gmail.com}
\keywords{ Power integral basis, index, Theorem of Ore, prime ideal factorization} \subjclass[2010]{11R04,
11R16, 11R21}
\date{}
\begin {abstract}
Let {$K$} be a  pure number field generated by  a complex root  of a monic irreducible polynomial $f(x)=x^{12}-m$ with  is a square free rational integer {$m\neq\mp 1$}. In this paper,   we prove that if  $m \equiv 2 \mbox{ or } 3\md4$ and $m\not\equiv \mp 1\md9$,  then the number field $K$  is monogenic.  But if {$m \equiv 1\md4$} or  $m\equiv  \mp 1\md9$,  then the number  field $K$ is not monogenic. 
\end{abstract}
\maketitle
 \section{Introduction}
  Let $K$ be a number field generated by $\al$ a complex root of a monic irreducible polynomial $f(x)\in
  \Z[x]$. { We denote its ring of integers by $\Z_K$.
  It is well know that  the ring $\Z_K$ is a free $\Z$-module of rank $n=[K:\Q]$. Let  $(\Z_K:\Z[\al])$ be the index of $\Z[\al]$ in $\Z_K$.   
  For any rational prime $p$, if $p$ does not divide the index $(\Z_K:\Z[\al])$, then thanks to a well-known { theorem of Dedekind},  the factorization of the ideal  $p \Z_K$ can be directly derived from the factorization of ${\overline f(x)}$ over $\F_p$, where ${\overline f(x)}$ is the reduction of $f(x)$ modulo $p$. { If $(\Z_K:\Z[\th])=1$ for  some $\th\in\Z_K$, i.e. if this index is not divisible by any  rational prime $p$, then  $(1,\th,\cdots,\th^{n-1})$ is a power integral bases of $\Z_K$. In this case,  the number field $K$ is said to
be monogenic and not monogenic otherwise.
The problem of testing the monogeneity of number fields and constructing power integral bases
have been intensively studied this last century, mainly by Ga\'al, Gy\"ory, Nakahara, Peth\"o,
Pohst and their research groups (see for instance \cite{AN, G, 2a, 10a, G19, 13a, GO, 17a, P}).
It is called a problem of Hasse to give an arithmetic characterization of those number fields which have a power integral basis \cite{10a, F4, G19, Ha, He, MNS, P}}.
{ In \cite{E07},   El Fadil  gave conditions for the existence of power integral bases of pure cubic fields in terms of the index form equation. In \cite{F4}, Funakura, studied the integral basis in pure quartic fields.  In \cite{GR4},  Ga\'al and  Remete, calculated the elements of index $1$, with coefficients with absolute value $<10^{1000}$ in the integral basis, of pure quartic fields generated by $m^{\frac{1}{4}}$ for {$1< m <10^7$}  and $m\equiv 2,3 \md4$.  In \cite{AN6}, Ahmad, Nakahara, and Husnine proved  that  if $m\equiv 2,3 \md4$ and  $m\not\equiv \mp1\md9$, then the sextic number field generated by $m^{\frac{1}{6}}$ is monogenic.
They also showed in \cite{AN},    that if $m\equiv 1 \md4$ and $m\not\equiv \mp1\md9$, then the sextic number field generated by $m^{\frac{1}{6}}$ is not monogenic.  Also, in \cite{E6}, based on prime ideal factorization, El Fadil showed that  if $m\equiv 1 \md4$ or $m\not\equiv 1\md9$, then the sextic number field generated by $m^{\frac{1}{6}}$ is not monogenic.
{Also, Hameed and Nakahara \cite{HN8}, proved that if $m\equiv 1\md{16}$, then the
octic number field generated by $m^{1/8}$ is not monogenic, but if $m\equiv 2,3 \md4$, 
then it is monogenic.}   In \cite{GR17}, Ga\'al  and  Remete, by applying the  explicit form of the index, they obtained new results on  monogeneity of the number  fields generated  by $m^{\frac{1}{n}}$, where $3\le n\le 9$.  While Ga\'al's and { Remete's} techniques are based on the index calculation,  Nakahara's methods are based on the existence of power relative integral bases of some special sub-fields. The goal of this paper is to study the monogeneity of pure number fields  defined by  $x^{12}-m$, where $m\neq 1$ is a square free integer. Our method is based on prime ideal factorization and is similar to that used in \cite{E6,  E24}}. 
 \section{Main results}
  { Our  main theorems below give a precise test for  { the
monogeneity of the number field $K=\Q(\al)$}, where $\al$ is a complex root
of an irreducible polynomial $f(x)=x^{12}- m \in \Z[x]$ with a square-free rational
integer $m\neq \mp 1$}.
\begin{thm}\label{pib}
Under the above hypothesis, if $m\equiv 2 \mbox{ or } 3 \md{4}$ and $m\not\equiv \mp 1\md9$, then {{ $\Z[\al]$ }} is the ring of integers of
$K$.
\end{thm}
\begin{thm}\label{npib}
Under the above hypothesis, if { $m\equiv   1\md4$} or $m\equiv  \mp 1\md9$, then $K$ is not monogenic.  
\end{thm}
\section{{Preliminaries}}
In order to {  prove}  Theorem \ref{pib} and  Theorem \ref{npib},  we recall some fundamental notions on  Newton polygon techniques. Namely, theorem of index and prime ideal factorization. 
 {{ Let  $\ol{f(x)}=\prod_{i=1}^r \ol{\ph_i(x)}^{l_i}$ modulo $p$ be the factorization of $\ol{f(x)}$ into powers of monic irreducible coprime polynomials of $\F_p[x]$. Recall that  Kummer showed that if $l_i=1$ for every $i=1,\dots,r$, then the factorization of $p\Z_K$ derived directely from the factorization of $\ol{f(x)}$. Namely,  $$p\Z_K=\prod_{i=1}^r {\p_i},\mbox{  where } \p_i=(p,\ph_i(\al)), \mbox{ for every }i=1,\dots,r.$$ Again a theorem due to  Dedekind says that: 
 \begin{thm}$($\cite[ Chapter I, Proposition 8.3]{Neu}$)$\\
 $$\mbox{If }   p \mbox{  does not divide the index } (\Z_K:\Z[\al]), \mbox{ then } p\Z_K=\prod_{i=1}^r \p_i^{l_i}, \mbox{ where every } \p_i=p\Z_K+\phi_i(\al)\Z_K$$  and the residue degree of $\p_i$ is $f(\p_i)={\mbox{deg}}(\phi_i)$.
 \end{thm}
 In order to apply this  theorem in an effective way one needs a
criterion to test  whether  $p$ divides  or not the index $(\Z_K:\Z[\al])$.}
  In $1878$, Dedekind  gave a criterion  to tests whether  $p$ divides or not  $(\Z_K: \Z[\al])$.
 {
  \begin{thm}\label{Ded}$($Dedekind's Criterion \cite[Theorem 6.1.4]{Co} and \cite{R}$)$\\
 For a number field $K$ generated by $\al$ a complex root of a monic irreducible  polynomial $f(x)\in \Z[x]$ and a rational prime integer $p$, let $\overline{f}(x)=\prod_{i=1}^r\overline{\ph_i}^{l_i}(x)\md{p}$  be the factorization of   $\overline{f}(x)$ in $\F_p[x]$, where the polynomials $\ph_i\in\Z[x]$ are monic with their reductions irreducible over $\F_p$ and GCD$(\overline{\ph_i},\overline{\ph_j})=1$ for every $i\neq j$. If we set
$M(x)=\cfrac{f(x)-\prod_{i=1}^r{\ph_i}^{l_i}(x)}{p}$, then $M(x)\in \Z[x]$ and the following statements are equivalent:
\begin{enumerate}
\item[1.]
$p$ does not divide the index $(\Z_K:\Z[\al])$.
\item[2.]
For every $i=1,\dots,r$, either $l_i=1$ or $l_i\ge 2$ and $\overline{\ph_i}(x)$ does not divide $\overline{M}(x)$ in $\F_p[x]$.
\end{enumerate}
\end{thm} }
 When Dedekind's criterion fails, that is,  $p$ divides the index $(\z_K:\z[\alpha])$ for all primitive elements $\al$ of $\z_K$ (see for example \cite{B2}),{  then} for such primes and number fields, it is not possible to obtain the prime ideal factorization of $p\z_K$ by {Dedekind's theorem}.
In 1928, Ore developed   an alternative approach
for obtaining the index $(\z_K:\z[\alpha])$, the
absolute discriminant, and the prime ideal factorization of the rational primes in
a number field $K$ by using Newton polygons (see {\cite{EMN, MN, O}}). For more details on Newton polygon techniques, we refer to \cite{El, GMN}}.
 For any  prime integer  $p$,{ let $\nu_p$ be the $p$-adic valuation of $\Q$, $\Q_p$ its $p$-adic completion, and $\Z_p$ the ring of $p$-adic integers. Let  also $\nu_p$ the Gauss's extension of $\nu_p$ to $\Q_p(x)$; $\nu_p(P)=\mbox{min}(\nu_p(a_i), \, i=0,\dots,n)$ for any polynomial $P=\sum_{i=0}^na_ix^i\in\Q_p[x]$ and extended by $\nu_p(P/Q)=\nu_p(P)-\nu_p(Q)$ for every nonzero polynomials $P$ and $Q$ of $\Q_p[x]$.} Let 
$\phi\in\z_p[x]$ be a   monic polynomial  {\it whose reduction} is irreducible  in
$\fp[x]$, let $\fph$ be 
the field $\frac{\fp[x]}{(\overline{\phi})}$. For any
monic polynomial  $f(x)\in \z_p[x]$, upon  the Euclidean division
 by successive powers of $\ph$, we  expand $f(x)$ as follows:
$$f(x)=\sum_{i=0}^la_i(x)\phi(x)^{i},$$ called    the $\phi$-expansion of $f(x)$
 (for every $i$, $deg(a_i(x))<
deg(\phi)$). 
The $\ph$-Newton polygon of $f(x)$ with respect to $p$, is the lower boundary convex envelope of the set of points $\{(i,\nu_p(a_i(x))),\, a_i(x)\neq 0\}$ in the Euclidean plane, which we denote by $\nph{f}$. For every $i\ne j=0,\dots,l$, let  $a_i=a_i(x)$ and $\mu_{ij}=\frac{\nu_p(a_{i})-\nu_p(a_{j})}{i-j}\in \Q$.  Then we obtain the following integers $0=i_0<i_1<\dots< i_r=l$ satisfying   
{$i_{j+1}=\mbox{ max }\{i=i_j+1,\dots l,\, \mu_{i_ji_{j+1}} \le \mu_{i_ji} \mbox{ and } \mu_{i_ji_{j+1}} \mbox{ is minimal}\}$}.  For every $j=1,\dots r$, let {$S_j$ be the segment joining the points $A_{j-1}=(i_{j-1}, \nu_p(a_{i_{j-1}}))$  and $A_j=(i_{j}, \nu_p(a_{i_{j}}))$} in the Euclidean plane. The segments $S_1,\dots,S_r$ are called the sides of the polygon $\nph{f}$.
For every $j=1,\dots,r$,
 the rational  number $\la_j=\frac{\nu_p(a_{i_{j}})-\nu_p(a_{i_{j-1}})}{i_{j}-i_{j-1}}$ is called the slope of $S_j$,
  $l(S_j)=i_{j}-i_{j-1}$ is its length,  
  and $h(S_j)=-\la_jl(S_j)$ is its height.{ In such a way    $l(S_j)$ is 
 the length of its projection to the $x$-axis,  $h(S_j)$  is the length of its projection to the $y$-axis, and $\nu_p(a_{i_{j}})=\nu_p(a_{i_{j-1}})+l(S_{j})\la_{j}$. }
  Remark that the $\ph$-Newton polygon of $f$, is the process of joining the segments $S_1,\dots,S_r$ ordered by   increasing slopes, which  can be expressed as $\nph{f}=S_1+\dots + S_r$.   
 For every side $S$ of the polygon $\npp{f}$ of length  $l(S)$ and height $h(S)$, {let $d(S)=$GCD$(l(S), h(S))$ be the ramification degree of $S$.
 The principal $\ph$-Newton polygon of ${f}$},
 denoted $\npp{f}$, is the part of the  polygon $\nph{f}$, which is  determined by joining all sides of negative  slopes.
  For every side $S$ of{$\npp{f}$}, with initial point $(s, u_s)$ and length $l$, and for every 
$0\le i\le l$, we attach   the following
{{\ residual coefficient} $c_i\in\fph$ as follows:
$$c_{i}=
\left
\{\begin{array}{ll} 0,& \mbox{ if } (s+i,{\it u_{s+i}}) \mbox{ lies strictly
above } S\\
\left(\dfrac{a_{s+i}(x)}{p^{{\it u_{s+i}}}}\right)
\,\,
\md{(p,\phi(x))},&\mbox{ if }(s+i,{\it u_{s+i}}) \mbox{ lies on }S.
\end{array}
\right.$$
where $(p,\phi(x))$ is the maximal ideal of $\z_p[x]$ generated by $p$ and $\ph$. 
%That means if $(s+i,{\it u_{s+i}}) \mbox{ lies on }S$, then $c_i=\overline{\dfrac{a_{s+i}(\beta)}{p^{{\it u_{s+i}}}}}$, where $\beta\in \C$ is a root of $\ph$.
Let $\la=-h/e$ be the slope of $S$, where  $h$ and $e$ are two positive coprime integers. Then  $d=l/e$ is the degree of $S$.  Notice that, 
the points  with integer coordinates lying{ on} $S$ are exactly $${(s,u_s),(s+e,u_{s}-h),\cdots, (s+de,u_{s}-dh)}$$. Thus, if $i$ is not a multiple of $e$, then 
$(s+i, u_{s+i})$ does not lie in $S$, and so $c_i=0$. Let
{$$f_S(y)=t_dy^d+t_{d-1}y^{d-1}+\cdots+t_{1}y+t_{0}\in\fph[y],$$}} called  
the residual polynomial of $f(x)$ associated to the side $S$, where for every $i=0,\dots,d$,  $t_i=c_{ie}$.
\begin{rem}
 Notice that as $t_d\neq 0$, deg$(f_S)=d$.\\
 Notice also that if $\nu_p(a_{s}(x))=0$, $\la=0$, and $\ph=x$, then $\fph=\F_p$ and for every $i=0,\dots,l$, $c_i=\overline{{a_{s+i}}} \md{p}$. Thus this notion of residual coefficient generalizes the reduction modulo the  maximal ideal $p$ and $f_S(y)\in\F_p[y]$ coincides with the reduction of $f(x)$ modulo  $(p)$.
    \end{rem}
    Let $\npp{f}=S_1+\dots + S_r$ be the principal $\ph$-Newton polygon of $f$ with respect to $p$.\\
     We say that $f$ is a $\ph$-regular polynomial with respect to $p$, if  $f_{S_i}(y)$ is square free in $\fph[y]$ for every  $i=1,\dots,r$. \\
      The polynomial $f$ is said to be  $p$-regular  if $\overline{f(x)}=\prod_{i=1}^t\overline{\ph_i}^{l_i}$ for some monic polynomials $\ph_1,\dots,\ph_t$ of $\Z[x]$ such that $\ol{\ph_1},\dots,\ol{\ph_t}$ are irreducible coprime polynomials over $\F_p$ and    $f$ is  a $\ph_i$-regular polynomial with respect to $p$ for every $i=1,\dots,t$.
 \smallskip
 
The  theorem of Ore plays  a fundamental key for proving our main Theorems:\\
  Let $\ph\in\Z_p[x]$ be a monic polynomial, with $\overline{\ph(x)}$ is irreducible in $\F_p[x]$. As defined in \cite[Def. 1.3]{EMN},   the $\ph$-index of $f(x)$, denoted by $ind_{\ph}(f)$, is  deg$(\ph)$ times the number of points with natural integer coordinates that lie below or on the polygon $\npp{f}$, strictly above the horizontal axis,{ and strictly beyond the vertical axis} (see $Figure\ 1$).
  
  \begin{figure}[htbp] 
\centering

\begin{tikzpicture}[x=1cm,y=0.5cm]
\draw[latex-latex] (0,6) -- (0,0) -- (10,0) ;

\draw[thick] (0,0) -- (-0.5,0);
\draw[thick] (0,0) -- (0,-0.5);

\node at (0,0) [below left,blue]{\footnotesize $0$};

\draw[thick] plot coordinates{(0,5) (1,3) (5,1) (9,0)};
\draw[thick, only marks, mark=x] plot coordinates{(1,1) (1,2) (1,3) (2,1)(2,2)     (3,1)  (3,2)  (4,1)(5,1)  };

\node at (0.5,4.2) [above  ,blue]{\footnotesize $S_{1}$};
\node at (3,2.2) [above   ,blue]{\footnotesize $S_{2}$};
\node at (7,0.5) [above   ,blue]{\footnotesize $S_{3}$};
\end{tikzpicture}
\caption{    \large  $\npp{f}$.}
\end{figure}

  Now assume that $\overline{f(x)}=\prod_{i=1}^t\overline{\ph_i}^{l_i}$ is the factorization of $\overline{f(x)}$ in $\F_p[x]$, where every $\ph_i\in\Z[x]$ is monic polynomial, with $\overline{\ph_i(x)}$ is irreducible in $\F_p[x]$, $\overline{\ph_i(x)}$ and $\overline{\ph_j(x)}$ are coprime when $i\neq j$ and $i, j=1,\dots,t$.
For every $i=1,\dots,t$, let  $N_{\ph_i}^+(f)=S_{i1}+\dots+S_{ir_i}$ be the principal  $\ph_i$-Newton polygon of $f$ with respect to $p$. For every $j=1,\dots, r_i$,  let $f_{S_{ij}}(y)=\prod_{k=1}^{s_{ij}}\psi_{ijk}^{a_{ijk}}(y)$ be the factorization of $f_{S_{ij}}(y)$ in $\F_{\ph_i}[y]$. 
  Then we have the following index theorem of Ore (see \cite[Theorem 1.7 and Theorem 1.9]{EMN}, \cite[Theorem 3.9]{El}, and{\cite[pp: 323--325]{MN}}).
 \begin{thm}\label{ore} $($Theorem of Ore$)$
 \begin{enumerate}
 \item
  $$\nu_p(ind(f))=\nu_p((\z_K:\z[\al]))\ge \sum_{i=1}^t ind_{\ph_i}(f).$$  The equality holds if $f(x)$ is $p$-regular.  %every $a_{ijk}= 1${for every $i=1,\dots,t$, $j=1,\dots,r_i$, and $k=1,\dots,s_{ij}$}.
\item
If  $f(x)$ is $p$-regular
% $a_{ijk}= 1${for every $i=1,\dots,t$, $j=1,\dots,r_i$, and $k=1,\dots,s_{ij}$}
, then
$$p\Z_K=\prod_{i=1}^t\prod_{j=1}^{r_i}
\prod_{k=1}^{s_{ij}}\p^{e_{ij}}_{ijk},$$ 
where{$e_{ij}=l_{ij}/d_{ij}$, $l_{ij}$ is the length of $S_{ij}$,  $d_{ij}$ is the ramification degree}
 of   $S_{ij}$, and $f_{ijk}=\mbox{deg}(\ph_i)\times \mbox{deg}(\psi_{ijk})$ is the residue degree of $\p_{ijk}$ over $p$.
 \end{enumerate}
\end{thm}
\section{{Proofs of main results}}
\begin{proof} of Theorem \ref{pib}.\\
 Since the discriminant of $f(x)$ is $\triangle(f)=\mp 12^{12}m^{11}$,
thanks to the formula linking the discriminant of $K$, the index, and
$\triangle(f)$, in order to prove that $\Z_K=\Z[\al]$ under the hypothesis of Theorem \ref{pib}, we need only to show that $p$
 does not divide the index $(\Z_K:\Z[\al])$  for every 
prime integer  dividing $2\cdot 3\cdot m$. 
Let $p$ be a prime
integer dividing  $2\cdot 3\cdot m$.  
\begin{enumerate}
\item
 $p$ divides  $m$. In this case $\overline{f(x)}=\ph^{12}$ in $\F_p[x]$, where $\ph=x$. As $\nu_p(m)=1$, $\nph{f}=S$ has a single side of height $1$, and so of degree $1$. Thus $f_S(y)$ is irreducible over $\fph$. By the first point of Theorem \ref{ore}, we get $\nu_p(ind(f))=ind_{\ph}(f)=0$; $p$ does not divide $(\Z_K:\Z[\al])$.
 \item
 $p=2$    and  $2$ does not divide $m$. Then $\overline{f(x)}=\ph_1^4\ph_2^4$ is the factorization of $\overline{f(x)}$ in $\F_2[x]$, where $\ph_1=(x-1)$ and $\ph_2=(x^2+x+1)$. 
 By considering $f(x+1)$, let $f(x)=\ph_1^{12}+\dots+495\ph_1^4+220\ph_1^3+66\ph_1^2+12\ph_1+1-m$ be the $\ph_1$-expansion of $f(x)$ and  $f(x)=\ph_2^6+(-6x+9)\ph_2^5-(25+5x)\ph_2^4+(24x+18)\ph_2^3-18x\ph_2^2+(4x-4)\ph_2+(1-m)$  the $\ph_2$-expansion of $f$.  So,  if $m\equiv 3\md4$; $\nu_2(1-m)=1$, then  for every $i=1,2$, $N_{\ph_i}(f)=S_i$ has a single side of degree $1$. Thus for every $i=1,2$,  $f_{S_i}(y)$ is irreducible over $\F_{\ph_i}$,  $ind_{\ph_1}(f)=ind_{\ph_2}(f)=0$. By  of Theorem \ref{ore}$(1)$, we conclude that $\nu_2(ind(f))=0$. That means that $2$ does not divide $(\Z_K:\Z[\al])$.
\item 
$p=3$    and  $3$ does not divide $m$. Then $\overline{f(x)}=(x^4-m)^3$ in $\F_3[x]$.\\
First case, $m=1\md3$. Then  $\overline{f(x)}=\ph_1^3\ph_2^3\ph_3^3$ in $\F_3[x]$, where $\ph_1=x-1$, $\ph_2=x+1$, and $\ph_3=x^2+1$.
Let $f(x)=\ph_1^{12}+\dots+495\ph_1^4+220\ph_1^3+66\ph_1^2+12\ph_1+1-m$ be the $\ph_1$-expansion of $f(x)$, $f(x)=\ph_2^{12}+\dots+495\ph_2^4-220\ph_2^3+66\ph_2^2-12\ph_2+1-m$  the $\ph_2$-expansion of $f(x)$, and $f(x)=\ph_3^{6}-6\ph_3^5+15\ph_3^4-20\ph_3^3+
15\ph_3^2-6\ph_3+1-m$ the $\ph_3$-expansion of $f$.  It follows that  if $m\not\equiv 1\md9$, then $\nu_3(1-m)=1$, and so  for every $i=1,2$, $N_{\ph_i}^+(f)=S_i$ has a single side of degree $1$. Thus for every $i=1,2,3$,  $f_{S_i}(y)$ is irreducible over $\F_{\ph_i}$ and $ind_{\ph_1}(f)=ind_{\ph_2}(f)=ind_{\ph_3}(f)=0$. By of Theorem \ref{ore}$(1)$,  $\nu_3(ind(f))=0$, and so $3$ does not divide $(\Z_K:\Z[\al])$.\\
Second case, $m=-1\md3$. Then  $\overline{f(x)}=\ph_1^3\ph_2^3$ in $\F_3[x]$, where $\ph_1=x^2+x-1$ and $\ph_2=x^2-x-1$.
Let $f(x)=\ph_1^6+(-6x+21)\ph_1^5+(-65x+125)\ph_1^4+(-256x+338)\ph_1^3+(-474x+468)\ph_1^2+(-420x+324)\ph_1+(-144x+89-m)$ be the $\ph_1$-expansion of $f$ and
$f(x)=\ph_2^6+(6x+21)\ph_2^5+(65x+125)\ph_2^4+(256x+338)\ph_2^3+(474x+468)\ph_2^2+(420x+324)\ph_2+(144x+89-m)$ the $\ph_2$-expansion of $f$.   It follows that  if $m\not\equiv -1\md9$, then {$\nu_3(89-m)=1$}. Thus, for every $i=1,2$, $N_{\ph_i}^+(f)=S_i$ has a single side of degree $1$. Thus for every $i=1,2$,  $f_{S_i}(y)$ is irreducible over $\F_{\ph_i}$ and  $ind_{\ph_1}(f)=ind_{\ph_2}(f)=0$. By Theorem \ref{ore},  $\nu_3(ind(f))=0$, and so $3$ does not divide $(\Z_K:\Z[\al])$.
\end{enumerate}
\end{proof}
{   If $p$ does not divide $(\Z_K:\Z[\al])$, then thanks to Dedekind's theorem, the factorization of $p\Z_K$ derived directly from the factorization of $\ol{f(x)}$ in $\F_p[x]$. If $p$  divides $(\Z_K:\Z[\al])$, then the following lemma allows the factorization of $p\Z_K$ into primes ideals of $\Z_K$ for $p=2,3$.  For every prime ideal factor $\p_{ijk}$, the residue degree $f_{ijk}$ is calculated. It
 plays a key role in  the proof of  Theorem \ref{npib}}.
\begin{lem}\label{fact2} 
{  
\begin{enumerate}
%\item
%{If $m\equiv 0 \md{2}$, then $2\Z_K=\p_{111}^{12}$, with $f_{111}=1$.}
%\item
%{If $m\equiv 3 \md{4}$, then $2\Z_K=\p_{111}^{4}\p_{211}^{4}$, with $f_{111}=1$ and $f_{211}=2$.}
\item
 If $m\equiv 1 \md{16}$, then {$2\Z_K=\p_{111}\p_{121}\p_{131}^2\p_{211}\p_{221}\p_{231}^2$} is the
factorization into product of prime ideals of $\Z_K$, { $f_{111}=f_{121}=f_{131}=1$ and $f_{211}=f_{221}=f_{231}=2$}.
$2$.
\item
 If $m\equiv 9 \md{16}$, then {$2\Z_K=\p_{111}\p_{121}^2\p_{211}\p_{212}\p_{221}^2$} is the
factorization into product of prime ideals of $\Z_K$, with
$f_{111}=2$, $f_{121}=f_{211}=f_{212}=f_{221}=2$.
\item
 {If $m\equiv 5 \md{8}$, then $2\Z_K=\p_{111}^2\p_{211}^2\p_{212}^2$ is the
factorization into product of prime ideals of $\Z_K$, with
$f_{111}=f_{211}=f_{212}=2$.}
%\item
%{ If $m\equiv 0 \md{3}$, then$3\Z_K=\p_{111}^{12}$, with residue degrees $f_{111}=1$.}
% \item
% {If $m\equiv 1 \md{3}$ and $m\not\equiv 1\md9$, then $3\Z_K=\p_{111}^3\p_{211}^3\p_{311}^3$, with residue degrees $f_{111}=f_{211}=1$ and $f_{311}=2$}.
%\item
%{ If $m\equiv -1 \md{3}$ and $m\not\equiv -1\md9$, then $3\Z_K=\p_{111}^3\p_{211}^3$, with residue degrees $f_{111}=f_{211}=2$}.
\item
 If $m\equiv 1 \md{9}$, then$3\Z_K=\p_{111}\p_{121}^2\p_{211}\p_{221}^2\p_{311}\p_{321}^2$, with residue degrees $f_{111}=f_{121}=f_{211}=f_{221}=1$ and $f_{311}=f_{321}=2$.
\item
 If $m\equiv -1 \md{9}$, then
 $3\Z_K=\p_{111}\p_{121}^2\p_{211}\p_{221}^2$ with  the same residue degree $2$ each factor.
\end{enumerate}}
\end{lem}
\begin{proof} 
%\begin{enumerate}
%\item
%$m\equiv 0 \md{2}$.{ As $m$ is square free,  then  $\nu_2(ind(f))=0$, and  so by Dedekind's theorem, the factorization of $2\Z_K$ derived directly from the factorization of $\ol{f(x)}$ modulo $2$. 
%Thus, $2\Z_K=\p_{111}^{12}$ with $f_{111}=1$.
%\item
%For the same reason, if $m\equiv 0\md3$, then   $\nu_3(ind(f))=0$, and  $3\Z_K=\p_{111}^{12}$ with $f_{111}=1$.
%\item
% Similarly, if $m\equiv 3 \md{4}$, then $\ol{f(x)}=(x-1)^4(x^2+x+1)^4$ in $\F_2[x]$. By Theorem \ref{pib} $\nu_2(ind(f))=0$, and so  $2\Z_K=\p_{111}^{4}\p_{211}^{4}$ with $f_{111}=1$ and $f_{211}=2$. 
%\item
% If
%$m\equiv 1\md3$ and $m\not\equiv  1\md9$, then $\ol{f(x)}=(x-1)^3(x+1)^3(x^2+1)^3$ in $\F_3[x]$. By Theorem \ref{pib} $\nu_3(ind(f))=0$, and so  $3\Z_K=\p_{111}^{3}\p_{211}^{3}\p_{311}^3$ with $f_{111}=f_{211}=1$ and $f_{311}=2$.
%\item
%If  $m\equiv -1\md3$ and $m\not\equiv  -1\md9$,  then $\ol{f(x)}=(x^2+x-1)^3(x^2-x-1)^3$ in $\F_3[x]$. By Theorem \ref{pib} $\nu_3(ind(f))=0$, and so  $3\Z_K=\p_{111}^{3}\p_{211}^{3}$ with $f_{111}=f_{211}=2$.
  If $m\equiv 1 \md{4}$, then $f(x)\equiv
\ph_1^4\ph_2^4\md{2}$, where $\ph_1=x-1$ and $\ph_2=x^2+x+1$.
 Let $f(x+1)=x^{12}+\dots+4955x^4+220x^3+66x^2+12x+1-m$. Then $f(x)=\ph_1^{12}+\dots+495\ph_1^4+220\ph_1^3+66\ph_1^2+12\ph_1+1-m$.  Let also $f(x)=\ph_2^6+(-6x+9)\ph_2^5-(25+5x)\ph_2^4+(24x+18)\ph_2^3-18x\ph_2^2+(4x-4)\ph_2+(1-m)$  be the $\ph_2$-expansion of $f$. It follows that:
    \begin{enumerate}
\item
{ If   $m\equiv 5 \md{8}$,  then $\nu_2(1-m)=2$ and 
{${N}_{\ph_1}^+(f)=S_{11}$} has a single  side joining the points $(0,2)$ and $(4,0)$. Thus its ramification   degree is $d_{11}=2$ and  $f_{S_{11}}(y)=y^2+y+1$ is irreducible over $\F_{\ph_1}\simeq \F_2$. Hence $\ph_1$ provides a single prime ideal $\p_{111}$ of $\Z_K$ lying above $2$ with  residue degree $f_{111}=2$. More precisely,  $2\Z_K=\p_{111}^2I$, where $I$ is an ideal of $\Z_K$, its factorization will be determined according to the Newton polygon $N_{\ph_2}^+(f)$.
 Again, as    $m\equiv 5\md{8}$, $\nu_2(1-m)=2$ and 
{${N}_{\ph_2}^+(f)=S_{21}$} has a single  side joining the points $(0,2)$ and $(4,0)$. Thus its ramification   degree is $d_{21}=2$ and   $f_{S_{21}}(y)=(1-j)y^2+jy+1=(1-j)(y-1)(y-j)$ in $\F_{\ph_2}[y]$, where $j=\ol{x}$ modulo $(2,x^2+x+1)$. Thus,  $I=\p_{211}^2\p_{212}^2$ with the same residue $2$.}
 \item
 If   $m\equiv 9 \md{16}$,  then $\nu_2(1-m)= 3$ and 
{${N}_{\ph_i}^+(f)=S_{i1}+S_{i2}$}{has two sides joining the points $(0,3),\ (2,1)$, and $(4,0)$ (see $Figure\ 2$). Thus $d_{i1}=2$, $d_{i2}=1$}, $f_{S_{11}}(y)=y^2+y+1$ and  $f_{S_{12}}(y)$ is of degree $1$. Thus they are irreducible over $\F_{\ph_1}\simeq \F_2$. It follows that  $\ph_1$ provides two distinct prime ideals $\p_{111}$ and $\p_{121}$ of $\Z_K$ lying above $2$ with residue degrees respectively are $f_{111}=2$ and $f_{121}=1$. More precisely, 
$2\Z_K=\p_{111}\p_{121}^2I$ (because $e_{12}=\frac{l_{12}}{d_{12}}=2$). For the factorization of $I$, first
{  deg$(f_{S_{22}})=d_{22}=1$ implies that  $f_{S_{22}}$}  is irreducible over $\F_{\ph_2}$. On the other hand, since $f_{S_{21}}(y)=jy^2+(j-1)y+1=j(y-1)(y-j^2)$ in $\F_{\ph_2}[y]$, where $j=\ol{x}$ modulo the ideal $(2,x^2+x+1)$, we conclude that $\ph_2$ provides three distinct   prime ideals $\p_{211}$, $\p_{212}$, and $\p_{221}$ of $\Z_K$ lying above $2$ with the same  residue degree  $2$. Especially,  $I=\p_{211}\p_{212}\p_{221}^2$ (because $e_{22}=\frac{l_{22}}{d_{22}}=2$).
\begin{figure}[htbp] 
\centering
\begin{tikzpicture}[x=1cm,y=0.5cm]
\draw[latex-latex] (0,6) -- (0,0) -- (10,0) ;
\draw[thick] (0,0) -- (-0.5,0);
\draw[thick] (0,0) -- (0,-0.5);
\node at (0,0) [below left,blue]{\footnotesize $0$};
\node at (4.2,0) [below left,blue]{\footnotesize $4$};
\node at (0,3.5) [below left,blue]{\footnotesize $3$};
\node at (2.2,0) [below left,blue]{\footnotesize $2$};
\node at (0,1.5) [below left,blue]{\footnotesize $1$};
%\node at (0,2) [below left,blue]{\footnotesize $2$};
%\node at (1,0) [below left,blue]{\footnotesize $1$};

\draw[thick] plot coordinates{(0,3) (2,1) (4,0)};
\draw[thick, only marks, mark=x] plot coordinates{(0,3) (1,2) (2,1)  (4,0)};

\node at (1,2.7) [above  ,blue]{\footnotesize $S_{i1}$};
\node at (3,1) [above   ,blue]{\footnotesize $S_{i2}$};
%\node at (7,0.5) [above   ,blue]{\footnotesize $S_{3}$};
\end{tikzpicture}
\caption{    \large  $N_{\ph_i}^+(f)$.}
\end{figure}
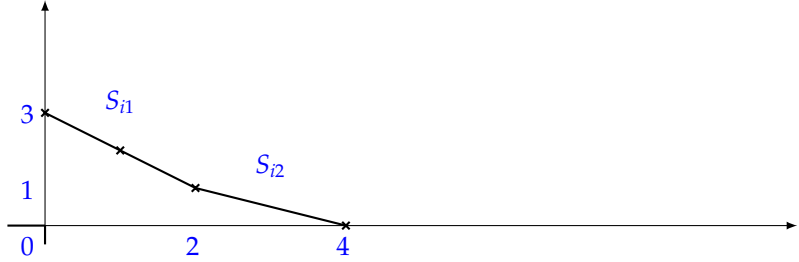
\item
{ If   $m\equiv 1 \md{16}$,  then $v_2=\nu_2(1-m)\ge 4$ and 
${N}_{\ph_i}^+(f)=S_{i1}+S_{i2}+S_{i3}$ has three sides joining the points $(0,v_2),\ (1,2),\ (2,1)$, and $(4,0)$ (see $Figure\ 3$). Thus all sides have the same ramification degree  $1$.}
It follows that  $\ph_i$ provides three distinct prime ideals $\p_{i11}$, $\p_{i21}$, and $\p_{i31}$ of $\Z_K$ lying above $2$ with the same residue degree  deg$(\ph_i)$ for every $i=1,2$. More precisely,  $2\Z_K=\p_{111}\p_{121}\p_{131}^2\p_{211}\p_{221}\p_{231}^2$ (because  $e_{i3}=\frac{l_{i3}}{d_{i3}}=2$).
\begin{figure}[htbp] 
\centering
\begin{tikzpicture}[x=1cm,y=0.5cm]
\draw[latex-latex] (0,6) -- (0,0) -- (10,0) ;
\draw[thick] (0,0) -- (-0.5,0);
\draw[thick] (0,0) -- (0,-0.5);
\node at (0,0) [below left,blue]{\footnotesize $0$};
\node at (4.2,0) [below left,blue]{\footnotesize $4$};
\node at (0,4.8) [below left,blue]{\footnotesize $v_2$};
\node at (2.2,0) [below left,blue]{\footnotesize $2$};
\node at (0,1.5) [below left,blue]{\footnotesize $1$};
\node at (0,2.5) [below left,blue]{\footnotesize $2$};
\node at (1.2,0) [below left,blue]{\footnotesize $1$};
\draw[thick] plot coordinates{(0,4.5) (1,2) (2,1) (4,0)};
%\draw[thick, only marks, mark=x] plot coordinates{(1,1) (1,2) (1,3) (2,1)(2,2)     (3,1)  (3,2)  (4,1)(5,1)  };
\node at (1,2.7) [above  ,blue]{\footnotesize $S_{i1}$};
\node at (2,1.8) [above   ,blue]{\footnotesize $S_{i2}$};
\node at (3.5,0.5) [above   ,blue]{\footnotesize $S_{i3}$};
\end{tikzpicture}
\caption{    \large  $N_{\ph_i}^+(f)$.}
\end{figure}
\end{enumerate}
For  $m\equiv 1\md9$, 
$f(x)\equiv
((x-1)^3((x+1)^3(x^2+1))^3\md{3}$. Let $\ph_1=x-1$, $\ph_2=x+1$,  $\ph_3=x^2+1$,
$F(x)=f(x+1)=x^{12}+\dots+220x^3+66x^2+12x+1-m$, and
$G(x)=f(x-1)=x^{12}-\dots-220x^3+66x^2-12x+1-m$. Then  $f(x)=\ph_1^{12}+\dots+220\ph_1^3+66\ph_1^2+12\ph_1+1-m$, $f(x)=\ph_2^{12}-\dots-220\ph_2^3+66\ph_2^2-12\ph_2+1-m$, and $f(x)=\ph_3^{6}-6\ph_3^5+15\ph_3^4-20\ph_3^3+15\ph_3^2-6\ph_3+1-m$. { As $v_3=\nu_3(m-1)\ge 2$, then ${N}_{\ph_i}^+(f)=S_{i1}+S_{i2}$ has two sides joining the points $(0,v_3)$, $(1,1)$, and $(3,0)$ (see $Figure\ 4$)}.
Thus 
  $3\Z_K=\p_{111}\p_{121}^2\p_{211}\p_{221}^2\p_{311}\p_{321}^2$ such that $f_{111}=f_{121}=f_{211}=f_{221}=1$ and $f_{311}=f_{312}=2$.
\begin{figure}[htbp] 
\centering

\begin{tikzpicture}[x=1cm,y=0.5cm]
\draw[latex-latex] (0,6) -- (0,0) -- (10,0) ;

\draw[thick] (0,0) -- (-0.5,0);
\draw[thick] (0,0) -- (0,-0.5);

\node at (0,0) [below left,blue]{\footnotesize $0$};
\node at (3,0) [below left,blue]{\footnotesize $3$};
\node at (0,2.8) [below left,blue]{\footnotesize $v_3$};
%\node at (2.2,0) [below left,blue]{\footnotesize $2$};
\node at (0,1.4) [below left,blue]{\footnotesize $1$};
%\node at (0,2.5) [below left,blue]{\footnotesize $2$};
\node at (1.2,0) [below left,blue]{\footnotesize $1$};

\draw[thick] plot coordinates{(0,2.5) (1,1)  (3,0)};
%\draw[thick, only marks, mark=x] plot coordinates{(1,1) (1,2) (1,3) (2,1)(2,2)     (3,1)  (3,2)  (4,1)(5,1)  };

\node at (0.7,2.2) [above  ,blue]{\footnotesize $S_{i1}$};
\node at (2,1) [above   ,blue]{\footnotesize $S_{i2}$};
%\node at (3.5,0.5) [above   ,blue]{\footnotesize $S_{i3}$};
\end{tikzpicture}
\caption{    \large  $N_{\ph_i}^+(f)$.}
\end{figure}
\smallskip

 Similarly, for   $m\equiv -1\md9$, 
 $\overline{f(x)}=\ph_1^3\ph_2^3$ in $\F_3[x]$, where $\ph_1=x^2+x-1$ and $\ph_2=x^2-x-1$.
Let{$f(x)=\ph_1^6+(-6x+21)\ph_1^5+(-65x+125)\ph_1^4+(-256x+338)\ph_1^3+(-474x+468)\ph_1^2+(-420x+324)\ph_1+(-144x+89-m)$} be the $\ph_1$-expansion of $f$ and
$f(x)=\ph_2^6+(6x+21)\ph_2^5+(65x+125)\ph_2^4+(256x+338)\ph_2^3+(474x+468)\ph_2^2+(420x+324)\ph_2+(144x+89-m)$ be the $\ph_2$-expansion of $f$. If $m\equiv -1\md9$, then {$v_3=\nu_3(\mp144x+89-m)= 2$, and so $N_{\ph_i}^+(f)=S_{i1}+S_{i2}$ has two sides joining the points $(0,2)$, $(1,1)$, and $(3,0)$ (see $Figure\ 5$)}.  Thus $3\Z_K=\p_{111}\p_{121}^2\p_{211}\p_{212}^2$ with  residue degree $2$ for every prime factor $\p_{ijk}$.
\begin{figure}[htbp] 
\centering

\begin{tikzpicture}[x=1cm,y=0.5cm]
\draw[latex-latex] (0,6) -- (0,0) -- (10,0) ;

\draw[thick] (0,0) -- (-0.5,0);
\draw[thick] (0,0) -- (0,-0.5);

\node at (0,0) [below left,blue]{\footnotesize $0$};
\node at (3.2,0) [below left,blue]{\footnotesize $3$};
\node at (0,2.6) [below left,blue]{\footnotesize $2$};
%\node at (2.2,0) [below left,blue]{\footnotesize $2$};
\node at (0,1.6) [below left,blue]{\footnotesize $1$};
%\node at (0,2.5) [below left,blue]{\footnotesize $2$};
\node at (1.2,0) [below left,blue]{\footnotesize $1$};

\draw[thick] plot coordinates{(0,2.2) (1,1)  (3,0)};
%\draw[thick, only marks, mark=x] plot coordinates{(1,1) (1,2) (1,3) (2,1)(2,2)     (3,1)  (3,2)  (4,1)(5,1)  };

\node at (0.7,2.2) [above  ,blue]{\footnotesize $S_{i1}$};
\node at (2,1) [above   ,blue]{\footnotesize $S_{i2}$};
%\node at (3.5,0.5) [above   ,blue]{\footnotesize $S_{i3}$};
\end{tikzpicture}
\caption{    \large  $N_{\ph_i}^+(f)$.}
\end{figure}
\end{proof}
 { 
Based on prime ideal factorization, the following lemma gives a sufficient condition for  the non monogeneity of $K$. Its proof is an immediate consequence of {Dedekind's} theorem. 
\begin{lem} \label{comindex}
 Let  $p$ be  rational prime integer and $K$  a number field. For every positive integer $f$, let $P_f$ be the number of distinct prime ideals of $\Z_K$ lying above $p$ with residue degree $f$ and $N_f$  the number of monic irreducible polynomials of  $\F_p[x]$ of degree $f$.
{ If $ P_f > N_f$ for some
positive integer $f$}, then for every generator $\th\in \Z_K$ of $K$, $p$ divide the index  $(\Z_K:\Z[\th])$.
\end{lem}}
\begin{proof} of Theorem \ref{npib}. { In every case, for an adequate prime integer $p$, we will show that $p$ divides the index  $(\Z_K:\Z[\th])$ for every generator $\th\in \Z_K$ of $K$, and so
 $K$ is not monogenic.}
 \begin{enumerate}
\item
 Assume that  $m\equiv 1 \md{8}$. By lemma \ref{fact2},  there are  at least { three } distinct prime ideals of $\Z_K$ lying above $2$, with residue degree $2$ each one. As there is only a unique monic irreducible polynomial of degree $2$ in $\F_2[x]$, namely $x^2+x+1$. Thus, by Lemma \ref{comindex}, $2$ divide the index  $(\Z_K:\Z[\th])$ for every generator $\th\in \Z_K$ of $K$. Hence $K$ is not monogenic.
  \item
  { Similarly, if $m\equiv 5 \md{8}$, then by Lemma \ref{fact2}, there are exactly { three } distinct prime ideals of $\Z_K$ lying above $2$, with residue degree $2$. As there is only a unique monic irreducible polynomial of degree $2$ in $\F_2[x]$, for every  $\th\in
\Z_K$,  $2$  divides the index $(\Z_K:\Z[\th])$. }
    \item
 If $m\equiv  1 \md{9}$, then by Lemma \ref{fact2}, there are  
 four distinct prime ideals of $\Z_K$ lying above $3$,  with residue degree $1$ each one. As there is only three monic irreducible polynomial of degree $1$ in $\F_3[x]$, by Lemma \ref{comindex},   $3$ divide the index  $(\Z_K:\Z[\th])$ for every generator $\th\in \Z_K$ of $K$. Hence $K$ is not monogenic.
 \item
  Similarly, if $m\equiv  -1 \md{9}$, by Lemma \ref{fact2}, there are 
 four distinct prime ideals of $\Z_K$ lying above $3$, with residue degree $2$ each one. As there is only three monic irreducible polynomial of degree $2$ in $\F_3[x]$, namely, $x^2+1$, $x^2+x-1$, and $x^2-x-1$, by Lemma \ref{comindex},   $3$ divide the index  $(\Z_K:\Z[\th])$ for every generator $\th\in \Z_K$ of $K$. Hence $K$ is not monogenic.
\end{enumerate}
\end{proof}
{\section*{Acknowledgements}
{The author is deeply  grateful to the anonymous referee whose valuable comments and suggestions have tremendously improved the quality of this paper. As well as for Professor Enric Nart who introduced him to Newton polygon techniques}.}
 
  \end{document}